\documentclass[12pt,a4paper,reqno,twoside]{amsproc} %reqno packed is used for making equation's numbers left to right
\usepackage{amsfonts, amssymb, amscd, amsmath, amsthm}
\usepackage{setspace}
\usepackage[utf8]{inputenc}
\usepackage{graphicx}
 \usepackage{float}
\usepackage{fullpage}
\usepackage{mathtools}
\newtheorem{theorem}{Theorem}[section]
\newtheorem{lemma}[theorem]{Lemma}
\newtheorem{proposition}[theorem]{Proposition}

\newtheorem{corollary}[theorem]{Corollary}
\newtheorem{definition}[theorem]{Definition}

\newtheorem{remark}[theorem]{Remark}

\newcommand{\R}{\mathbb{R}}
\newcommand{\C}{\mathbb{C}}
\newcommand{\T}{\mathcal{T}}
\newcommand{\B}{\mathcal{B}}
\newcommand{\N}{\mathbb{N}}
\newcommand{\M}{\mathbb{M}}
\newcommand{\A}{\mathcal{A}}

\usepackage{hyperref}
\hypersetup{
    colorlinks=true,
    linkcolor=blue,
    filecolor=blue,%magenta      
  urlcolor=blue,
    pdftitle={Sharelatex Example},
    bookmarks=true,  
    pdfpagemode=FullScreen
   }
\DeclareMathOperator{\lead}{lead}

\DeclareMathOperator{\co}{co}
\DeclareMathOperator{\ext}{ext}

\usepackage{comment}
\title{Quantization of $A_{0}(K)$-Spaces}
\author{Anindya Ghatak and Anil Kumar Karn}
\address[]{School of Mathematical Sciences, National Institute of Science Education and Research - Bhubaneswar, P.O.: Bhimpur-Padanpur,
Via- Jatni, Dist- Khurda, Odisha-752050, India}
\date{}
\begin{document}
\vspace{3in}
\vspace{3in}
\maketitle 
\begin{abstract}
	In this paper, we study $L^1$-matrix convex sets $\{K_{n}\}$ in $*$-locally convex spaces and show that every C$^*$-ordered operator space is complete isometrically, completely isomorphic to $\{A_{0}(K_{n}, M_{n}(V))\}$ for a suitable $L^1$-matrix convex set $\{K_{n}\}$.  Further, we generalize the notion of regular embedding of a compact convex set to $L^{1}$-regular embedding of $L^{1}$-matrix convex set. Using $L^{1}$-regular embedding of $L^{1}$-convex set, we find conditions under which $A_{0}(K_{n}, M_{n}(V))$ is an abstract operator system.  
\end{abstract}
%% maketitle must follow the abstract.
\maketitle 

\section{Introduction}
Kadison's realization of the self-adjoint part of an unital C$^{\ast}$-algebra $\A$ as the space of continuous real valued affine functions on the state space of $\A$ is one of the early corner-stone in the order-theoretic Functional Analysis \cite{KR51}. In this seminal paper of 1951, he observed that the same result holds for the self-adjoint part of any unital self-adjoint subspace of $\A$ (that is, an operator system in $\A$). In particular, he showed that the self-adjoint part of an operator system is an order unit space. Let $K$ be a compact and convex set in a locally convex space $X$ and let $A(K)$ denote the space of all real valued continuous affine functions on $K$. Then $A(K)$ is an order unit space. In 1968, Asimov introduced the notion of an universal cap (say, $K$) of a cone in a real ordered vector space and studied $A_0(K)$ as a non-unital prototype of $A(K)$\cite{ASMV68}. (See also, \cite{NG68}.)  

A nice duality theory for ordered Banach spaces was laid down during 1950's and 60's in the works of Bonsall, Edwards, Ellis, Asimov and Ng and many others. (See \cite{JEM70} and \cite{WK73} and references therein.) However, the functional representation theorem of Kadison (and the work that followed) was limited to self-adjoint elements only. Subsequently, the order theoretic Functional Analysis was limited to only real scalars. After a long gap, in 1976, Effros observed the following relation between the norm of an arbitrary element of a C$^*$-algebra $\A$ and the order structure in $M_2(\A)$: 
$$\Vert a \Vert \le 1 ~ \textrm{if and only if} ~ \begin{bmatrix} 1  &  a \\ a^* & 1 \end{bmatrix}.$$
Following this, in 1977, Choi and Effros introduced matrix ordered spaces and proved a generalization of Kadison's order unit spaces. More precisely, they proved that every operator system is exactly a matrix order unit space (definition is given below). This theory is also known as a beginning of quantization of Functional Analysis. In this sense, the Choi-Effros realization of an operator system as a matrix order unit space is a quantization of order unit space. A quantization of $A(K)$ appeared in the work of Webster and Winkler \cite{CS99} where they proved an operator space version of the Krein-Milman theorem. 

In 2001, the second author introduced the notion of C$^*$-ordered operator space. It was proved that an (abstract) C$^{*}$-ordered operator space is precisely an (abstract) $*$-operator space which can be ``order embedded'' in a C$^*$-algebra \cite{KARN11}. In this paper, we prove that a `quantized' functional representation of C$^*$-ordered operator spaces. Webster and Winkler's quantized functional representation of abstract operator systems relies on matrix convex sets. However, we could not construct a `matricial version of universal cap in this context of matrix convex sets which is needed to study non-unital case. To overcome this problem, we consider matrix (Choi-Effros) duality and introduce the notion of $L^{1}$-matrix convex sets. We prove that if $V$ is a C$^*$-ordered operator space and $Q_{n}=\{f\in M_{n}(V^*)^+: \Vert f\Vert\leq 1\}$ (in the matrix duality), then $\{Q_{n}(V)\}$ is an $L^1$-matrix convex set, and $V$ is complete isometric, completely order isomorphic to $\{A_{0}(Q_{n}(V), M_{n}(V^*))\}$. Conversely, we show that if $\{K_{n}\}$ is an $L^{1}$-matrix convex set in a $*$-locally convex space $V$, then $\{A_{0}(K_n, M_{n}(V))\}$ in a C$^*$-ordered operator space. Further we study additional properties of an $L^1$-matrix convex set $\{K_{n}\}$ in order to relate it to abstract operator system. 

\section{C$^*$-ordered Operator Spaces}
If $V$ is a complex $*$-vector space, we denote $V_{sa}$ to be the set of \emph{self-adjoint} elements of $V.$ A set $V^+\subseteq V_{sa}$ is \emph{cone} in $V$ if $V^+$ is additive and positive homogeneous (that is, $\lambda v\in V^{+}$ whenever $\lambda \in \R, v \in V^{+}$). In this case, we say that $(V, V^+)$ is complex ordered vector space. We write $u\leq v,$ or equivalently $v\geq u$ if and only if $v-u\in V^{+}.$ We say that $V^{+}$ is \emph{proper} if $V^{+}\cap -V^{+}=\{0\}$ and generating if $V^{+}$ spans whole $V.$ An element $e\in V^{+}$ is an \emph{order unit} for $V$ if for each $v\in V,$ there is a $t>0$ such that $-te\leq v\leq te.$ The cone $V^{+}$ is an \emph{Archimedean}, if for each $v_{0}\in V^{+}$ with $-t v_{0}\leq v$ for all $t>0$ implies $v\in V^{+}.$ If $V$ has an order unit $e,$ it is sufficient to consider $v_{0}=e.$ Let $(V_{i}, V_{i}^{+})$ be the complex ordered vector space for $i=1,2$  and let $\phi: V_{1}\mapsto V_{2}$ be a self-adjoint linear map. We say that $\phi$ is \emph{positive} if $\phi(V_{1}^{+})\subseteq V_{2}^{+}$. Moreover $\phi$ is called an \emph{order isomorphism} if $\phi$ is an isomorphism and $\phi, \phi^{-1}$ are positive.

We know that if $V$ is a $*$-vector space, then $M_{n}(V)$ is also $*$-vector space with $[v_{i,j}]^{*}=[v_{j,i}^{*}].$ A complex $*$-vector space $V$ is called a \emph{matrix ordered} if $M_{n}(V)^+\subseteq M_{n}(V)_{sa}$ is a cone for each $n$ such that $\gamma^{*} M_{m}(V)^+ \gamma \subseteq M_{n}(V)^+$ whenever $\gamma\in\M_{m,n}.$ A matrix ordered  space $(V, \{M_{n}(V)^+\})$ with an order unit $e$ is called a \emph{matrix order unit space} if $V^+$ has an proper and $M_{n}(V)^{+}$ is Archimedean for each $n$ \cite[Choi, Effros]{CE77}. It may be noted that $V$ is matrix ordered then its matrix dual $V^*$ is also matrix ordered where $M_{n}(V^*)^+=\{f\in M_{n}(V^{*})_{sa}: f(v)\geq 0 ~~ \forall v\in M_{n}(V)^{+}\}.$ 
\begin{theorem}\cite{CE77} 
	Let $(V,\{M_{n}(V)^+\},e)$ be an abstract operator system. Then there is a Hilbert space $H$ and a concrete operator system $\mathcal{S}\subseteq \B(H)$ and a complete order isomorphism $\phi: V\mapsto \mathcal{S}$ such that $\phi(e)=I,$ where $I$ is the identity operator on $H$.
\end{theorem}
An \emph{$L^{\infty}$-matricially normed space} $(V, \{\Vert\cdot\Vert_{n}\})$ is a complex vector space $V$ together with sequence of norms $\{\Vert\cdot\Vert_{n}\}$ such that
\begin{enumerate}
	\item $(M_{n}(V),\Vert\cdot\Vert_{n})$ is a normed linear space for all $n;$
	\item $\Vert v\oplus w\Vert_{n+m}=\max\{\Vert v \Vert_{n},\Vert
	w\Vert_{m}\};$
	\item $\Vert\alpha v\beta\Vert_{n}\leq \Vert\alpha\Vert\Vert
	v\Vert_{n}\Vert\beta\Vert.$
\end{enumerate}
An $L^{\infty}$-matricially normed space is called  an \emph{abstract operator space}.
Every abstract operator space is completely isometric to some concrete operator space of $\B(H)$ for some Hilbert space $H$ \cite{R98}.

Next, an \emph{$L^{1}$-matricially normed space} $(V, \{\Vert\cdot\Vert_{n}\})$ is a complex vector space together with sequence of matrix norms $\{\Vert\cdot\Vert_{n}\}$ such that
\begin{enumerate}
	\item $(M_{n}(V),\Vert\cdot\Vert_{n})$ is a normed linear space for each $n;$
	\item $\Vert v\oplus w\Vert_{n+m}=\Vert v \Vert_{n} + \Vert
	w\Vert_{m};$
	\item $\Vert\alpha v\beta\Vert_{n}\leq \Vert\alpha\Vert\Vert
	v\Vert_{n}\Vert\beta\Vert.$
\end{enumerate}
We know from \cite[Theorem 5.1]{R98} that if $V$ is an $L^{\infty}$-matricially normed space, then its matricial dual $V^*$ is an $L^1$-matricially normed space with scalar pairing is given by
\begin{align}
\langle [v_{i,j}], [f_{i,j}]\rangle(=[f_{i,j}]([v_{i,j}]))=\sum_{i,j=1}^{n}f_{i,j}(v_{i,j}).
\end{align}\label{e4}
\begin{definition}[C$^{\ast}$-ordered operator space]\cite{KARN11} 
	A $*$-vector space $V$ together with a matrix norm $\{ \Vert\cdot\Vert_n \}$ and a matrix order $\{ M_n(V)^+ \}$ is said to be a C$^{\ast}$-ordered operator space if $(V, \{ \Vert\cdot\Vert_n \})$ is an abstract operator space, $V^+$ is proper and if for each $n\in \N,$ the following conditions hold:
	\begin{enumerate}
		\item $*$ is an isometry on $M_{n}(V);$ 
		\item $M_{n}(V)^+$ is closed; 
		\item  $\Vert f\Vert_{n}\leq \max\{\Vert g\Vert_{n},\Vert h\Vert_{n}\},$ whenever $f\leq g\leq h$ with $f,g,h\in M_{n}(V)_{sa}.$
	\end{enumerate}
\end{definition}
We know that every matrix order unit space is a C$^{\ast}$-ordered operator space. 
\begin{theorem}\cite{KARN11}
	Let $(V, \{M_{n}(V)^+\}, \{\Vert\cdot\Vert_{n}\})$ be a C$^{\ast}$-ordered operator space. Then there exist a completely order isomerty
	$\phi:V\mapsto \A$ for some C$^{\ast}$-algebra $\A.$
\end{theorem}
Let $V$ be a C$^{\ast}$-ordered operator space. Then its matricial dual $V^*$ is an $L^{1}$-matricially normed space with an involution such that $*$ is isometry on $M_{n}(V^*)$ and $(V^*, \{ M_n(V^*)^+ \})$ is a matrix ordered space such that $M_n(V^*)^+$ is norm closed for each $n$. We put 
$$Q_{n}(V)=\{f\in M_{n}(V^*): f\geq 0,\Vert f\Vert_{n}\leq 1 \}.$$ 
Then $Q_n(V)$ is called the \emph{quasi state} of $M_{n}(V).$ We know that $Q_{n}(V)$ is compact convex set with respect to $w^*$-topology. 
\begin{lemma}\label{d4}
	$M_{n}(V^{*})_{sa}\cap M_{n}(V^*)_{1}=\co(Q_{n}(V)\cup -Q_{n}(V)).$ 
\end{lemma}
\begin{proof} 
	Let $f \in M_n(V^*)_{sa}.$ Then by \cite[Theorem 2.2]{KARN10}, there are $g, h \in M_n(V^*)^+$ such that $f = g - h$ and $\Vert f \Vert_n=\Vert g \Vert_n + \Vert h \Vert_n$. Therefore $M_n(V^*)_{sa} \cap M_n(V)_1 \subseteq \co(Q_n(V) \cup (-Q_n(V))).$ Since $\pm Q_n(V) \subseteq M_n(V^*)_{sa} \cap M_n(V)_1$ and $M_n(V^*)_{sa} \cap M_n(V)_1$ is convex, we have $\co(Q_n(V)\cup (-Q_n(V)))\subseteq M_n(V^*)_{sa}\cap M_n(V)_1.$
\end{proof} 
Now, we describe a `quantized' functional representation of a C$^*$-ordered operator space $V$. For $v\in V$, define $\check{v}:V^*\mapsto \C$ given by $\check{v}(f)=f(v)$ for all $f \in V^*$. Then $\check{v}$ is a $w^*$-continuous linear functional on $V^*$. We set $\check{v}|_{Q(V)} = \hat{v}$. Then $\hat{v}: Q_{1}(V)\mapsto \C$ is an affine, $w^*$-continuous map on $Q_{1}(V)$ such that $\hat{v}(0)=0$. Note that $\check{v}$ is the unique extension of $\hat{v}$ on $V^*$ as a $w^*$-continuous linear functional for $V^{*+}= \cup_{k\in \N}kQ_{1}(V)$ and $V^{*+}$ spans $V^*$. We write $A_{0}(Q_{1}(V), V^*)$ for the space of all $w^*$-continuous affine mappings from $Q_{1}(V)\mapsto \C$ vanishing at $0$ and having a unique $w^*$-continuous linear extension to $V^*$. Then $v \mapsto \hat{v}$ determines a linear $*$-isomorphism from $\Gamma: V \to A_{0}(Q_{1}(V), V^*)$. Further as $w^*$-dual of $V^*$ is identified with $V$, we may conclude that $\Gamma$ is surjective. For $v\in V$, set $(\check{v})^*= \check{(v^*)}$ so that 
$$
(\check{v})^*(f)=f(v^*)= \overline{f^*(v)}= \overline{\check{v}(f^*)}
$$
for all $f\in V^*$. In particular for $v\in V_{sa}$ and $f\in V^*_{sa},$ $(\check{v})^*(f) = \check{v}(f)\in \R$. Similarly, if $v\in V^+$ and $f\in V^{*+},$ then $\check{v}(f)\geq 0$. In fact, as $v\in V^+$ if and only if $f(v)\geq 0$ for every $f\in Q(V)$, we may conclude that
\begin{align*} 
\Gamma(V^+)&=\{\phi\in A_{0}(Q_{1}(V), V^*)_{sa}: \phi(f)\geq 0 ~\forall f\in Q_{1}(V)\}\\
&:=A_{0}(Q_{1}(V), V^{*})^+.
\end{align*}
In other words, $\Gamma$ is an order isomorphism. Now using matrix duality, we may further conclude that 
$$\Gamma_{n}: M_{n}(V)\mapsto A_{0}(Q_{n}(V), M_{n}(V^*))$$
given by 
$$\Gamma_{n}([v_{i,j}])= [\widehat{v_{i,j}}],~~[v_{i,j}]\in M_{n}(V)$$
is a surjective order isomorphism for each $n\in \N$. Now, if we identify $A_{0}(Q_{n}(V), M_{n}(V^*))$ with $M_{n}(A_{0}(Q_{1}(V), V^*))$ for each $n \in \N$, then $\Gamma : V\mapsto A_{0}(Q_{1}(V), V^*)$ is a surjective order isomorphism.

Next, we describe a norm on $A_{0}(Q_{n}(V), M_{n}(V^*))$. Let $F \in A_{0}(Q_{n}(V), M_{n}(V^*)).$ Then there is an unique $v\in M_{n}(V)$ such that $F= \Gamma_{n}(v) = \widehat{v}.$ We define 
\begin{align} 
\Vert F\Vert_{\infty,n}=
\sup\left\{\left\vert
\begin{bmatrix}
0 &\widehat{v}\\
\widehat{v^*}&0
\end{bmatrix}(f)\right\vert: f\in Q_{2n}(V)
\right\}.
\end{align}
As $v\in M_{n}(V),$ we have
$
\begin{bmatrix}
0 & v\\
v^*&0
\end{bmatrix}\in M_{2n}(V)_{sa}.$ 
Since $*$ is isometry in $V$, using Lemma \ref{d4}, we have 
\begin{eqnarray*}
	\left\vert\left\vert \begin{bmatrix}
		0  &v\\
		v^*&0
	\end{bmatrix}\right\vert\right\vert_{n} &=&
	\sup\left\{\left\vert
	\begin{bmatrix}
		0 &\widehat{v}\\
		\widehat{v^*}&0
	\end{bmatrix}(f)\right\vert: f\in M_{2n}(V^*)_{sa}\cap\M_{n}(V^*)_{1}
	\right\} \\
	&=&\sup\left\{\left\vert
	\begin{bmatrix}
		0 &\widehat{v}\\
		\widehat{v^*}&0
	\end{bmatrix}(f)\right\vert: f\in Q_{n}(V)
	\right\}.
\end{eqnarray*}
Also as $\ast$ is isometry and $\{\Vert\cdot\Vert_{n}\}$ is $L^{\infty}$-matrix norm, we
have $\Vert v \Vert_{n}= \left \Vert 
\begin{bmatrix}
0 & v\\
v^*&0
\end{bmatrix}\right\Vert_{2n}$ so that
$\Vert v \Vert_{n}= \Vert \widehat{v}\Vert_{\infty,n}.$ 
We summarize this discussion in the following: 
\begin{theorem} 
	$\Gamma: V \mapsto A_{0}(Q_1(V), V^*)$ is a surjective, complete isometric, complete order isomorphism. 
\end{theorem}
\section{Convexity of Matricial Quasi States} 
Let $V$ be a C$^*$-ordered operator space. In Section 2, we observed that $V$ has a quantized functional representation over the `matricial quasi-state' $\{ Q_n(V) \}$ of $V$. In this section, we shall discuss the matricial convexity of $\{ Q_n(V) \}$ in order to characterize the `quantized' functional representation of a C$^*$-ordered operator space. First, let us note that $\{ Q_n(V) \}$ fails to be a matrix convex set.
\begin{definition} 
	Let $W$ be a vector space. A collection $\{K_{n}\}$ with $K_{n} \subseteq M_{n}(W)$ is called a matrix convex set if $ \sum_{i=1}^{k} \gamma_{i}^{*}w_{i} \gamma_{i} \in K_{m}$ whenever $w_{i}\in M_{n_{i}}(W)$ and matrices $\gamma_{i} \in M_{n_{i},m}(1 \le i \le k)$ satisfy $\sum_{i=1}^{k} \gamma_{i}^{*}\gamma_{i}=I_{m}.$ 
\end{definition}
We recall that if $V$ is an $L^{\infty}$-matricially normed space (abstract operator space), then $\{M_{n}(V)_{1}\}$ is a matrix convex set. Here $E_1 := \{ v \in E: \Vert v \Vert \le 1 \}$. Hence if $V$ is a C$^{\ast}$-ordered operator space, then $\{M_{n}(V)_{1}^{+}\}$ is a matrix convex set. However, in this case, $\{Q_{n}(V)\}$ is not a matrix convex set. To see this, let $f\in Q(V)$ with $\Vert f \Vert=1$. Then $\Vert f \oplus f \Vert_{2} = 2$ so that $f \oplus f \notin Q_{2}(V).$ Whereas in a matrix convex set $\{ K_{n} \},$  we have $K_{1} \oplus K_{1} \subset K_{2}$ so that $\{Q_{n}(V)\}$ is not a matrix convex set. In fact, $\{Q_{n}(V)\}$ exhibit another kind of matricial convexity. 
\begin{definition}
	Let $K$ be a compact convex set in a locally convex set $V$ such that $0\in {\ext(K)}.$ An element
	$k\in K$ will be called a \emph{lead point} of $K$ $(k\in \lead(K))$ if $k=\alpha k_{1}$ for some $k_{1}\in K$ with $\alpha\in
	[0,1],$ then $\alpha=1.$
\end{definition}
We observe that $\ext(K)\setminus \{0\}\subseteq \lead(K).$
\begin{proposition}\label{c4}
	For each $k\in K\setminus \{0\}.$ There is a unique $\alpha \in (0,1]$ and $k_{1}\in \lead(K)$ such that
	$k=\alpha k_{1}.$
\end{proposition}

\begin{proof}
	Without any loss of generality, we may assume that $k\in K\setminus \lead(K).$ Then there is an
	$\alpha\in (0,1]$ and $k\in K$ such that $k=\alpha k_{1}.$ Thus the set $R_{K} = \{\beta \geq 1: \beta k \in K \}$ is non-empty. As $K$ is a compact $R_{K}$ is bounded and we have $\beta _{0}=\sup R_{K}\in R_{K}.$ Let
	$k_0 =\beta_{0}k\in K$ so that $k =\beta_{0}^{-1}k_{0}.$ We show that $k_{0}\in \lead(K).$ 
	If possible, assume that $k_0 \not\in \lead(K).$ Then by the definition of lead, there is a $\beta\in (0,1)$ and  
	$k^{\prime}\in K$ such that
	$k_{0}=\beta k^{\prime}.$ But, then $\beta^{-1}\beta_{0}k\in K,$ where $\beta^{-1}\beta_{0}>\beta_0,$ 
	which contradict $\beta_{0}=\sup R_{K}.$ Thus $k_{0}\in \lead(K).$ Next we prove the uniqueness of $k_0.$
	Let $k=\alpha_{1}k_{1}$ for some $k_{1} \in \lead(K)$  and 
	$\alpha_{1} \in (0,1].$ We see that $k_{1}=\alpha_{1}^{-1}\beta_{0} k_0.$ Thus
	$\alpha_{1}^{-1}\beta_{0}=1$ and hence $\alpha_{1}=\beta_{0}, k_{1}=k_{0}.$
\end{proof}
A \emph{$*$-locally convex space} is a locally convex space together with
$*$-operation which is a homeomorphism. In this case $M_{n}(V)$ is also a $*$-locally convex
space with respect to product topology.
\begin{definition}[$L^{1}$-matrix convex set]\label{b6} Let $V$ be a $*$-locally convex
	space. Let $\{K_{n}\}$ be a collection of compact convex sets $K_{n}\subseteq
	M_{n}(V)_{sa}$ with $0\in \partial_{e}(K_{n})$ for all $n.$  
	Then the collection of sets $\{K_{n}\}$ is called
	an \emph{$L^{1}$-matrix convex set} if the following conditions hold:
	
	\begin{itemize}
		\item [${\bf L_{1}}$] If $u\in K_{n}$ and $\gamma_{i}\in \M_{n, n_{i}}$ such that
		$\sum_{i=1}^{k}\gamma_{i}\gamma_{i}^{*}\leq I_{n},$ then
		$\oplus_{i=1}^{k}\gamma_{i}^*u\gamma_{i}\in K_{\sum_{i=1}^{k} n_{i}}.$
		\item [${\bf L_{2}}$] If $u\in K_{2n}$ so that 
		$u=\begin{bmatrix}
		u_{11}& u_{12} \\
		u_{12}^*& u_{22} 
		\end{bmatrix}$ 
		for some $u_{11}, u_{22}\in K_{n}$  and $u_{12}\in M_{n}(V),$
		then $u_{12}+u_{12}^{*}\in \co(K_{n}\cup-K_{n}).$
		\item [${\bf L_{3}}$] Let $u\in K_{m+n}$ with 
		$u=\begin{bmatrix}
		u_{11}& u_{12} \\
		u_{12}^*& u_{22} 
		\end{bmatrix}$ 
		so that $u_{11}\in K_{m}, u_{22}\in K_{n}$ and $u_{12}\in M_{m,n}(V)$ and if 
		$u_{11}=\alpha_{1}\widehat{u_{11}}, u_{22}=\alpha_{22}\widehat{u_{22}}$ with 
		$\widehat{u_{11}}\in \lead(K_{m}), \widehat{u_{22}}\in\lead(K_{n}),$
		then $\alpha_{1}+\alpha_{2}\leq 1.$
	\end{itemize}
\end{definition} 
The following result shows that $\{Q_{n}(V)\}$ is an $L^1$-matrix convex set. For each $n \in \N$, put $S_{n}(V) = \{ f \in Q_n(V): \Vert f \Vert_n = 1 \}.$ 
\begin{proposition} \label{b4} 
	Let $V$ be a C$^{\ast}$-ordered operator space and let $f\in Q_{m+n}(V)$ so that $f=\begin{bmatrix} f_{11}& f_{12} \\ f_{12}^*& f_{22} \end{bmatrix}$ for some $f_{11}\in M_{m}(V^*)^+$, $f_{22}\in M_{n}(V^*)^+$ and $f_{12} \in M_{m,n}(V^*)$. Then
	\begin{enumerate} 
		\item $f_{11} \in Q_{m}(V)$ and $f_{22} \in Q_{n}(V)$; 
		
		\item $\begin{bmatrix} f_{11}& e^{i\theta} f_{12} \\ e^{-i\theta} f_{12}^*& f_{22} \end{bmatrix} \in Q_{m+n}(V)$ for any $\theta\in \R;$ 
		
		\item $\left\vert\left\vert \begin{bmatrix} 0& f_{12} \\ f_{12}^*& 0 \end{bmatrix} \right\vert\right\vert_{m+n} \leq \left\vert\left\vert \begin{bmatrix} f_{11}& f_{12} \\ f_{12}^*& f_{22} \end{bmatrix} \right\vert\right\vert_{m+n}, \left\vert\left\vert \begin{bmatrix} f_{11}& 0 \\ 0 & f_{22} \end{bmatrix} \right\vert\right\vert_{m+n} \leq \left\vert\left\vert \begin{bmatrix} f_{11}& f_{12} \\ f_{12}^*& f_{22} \end{bmatrix} \right\vert\right\vert_{m+n};$ 
		
		\item If $m = n$, then $f_{12}+f_{12}^{*}\in \co(Q_{n}(V)\cup (-Q_{n}(V))).$ 
		
		\item  Let $f\in Q_{n}(V)$ and let $\gamma_{i}\in \M_{n,n_{i}}$ such that $\sum_{i=1}^{k} \gamma_{i} \gamma_{i}^* \leq I_{n}$. Then $\oplus_{i=1}^{k}\gamma_{i}^*f\gamma_{i}\in Q_{\sum_{i=1}^{k} n_{i}}(V).$ 
		
		\item Let $f\in Q_{m+n}(V)$ with $f=\begin{bmatrix} f_{11}& f_{12} \\ f_{12}^*& f_{22} \end{bmatrix}$ so that $f_{11}\in Q_{m}(V), f_{22}\in Q_{n}(V)$ and $f_{12}\in M_{n}(V)$ and let $f_{11}= \alpha_{1} \widehat{f_{11}}, f_{22} = \alpha_{2} \widehat{f_{22}}$ with $\widehat{f_{11}} \in S_{m}(V), \widehat{f_{22}} \in S_{n}(V).$ Then $\alpha_{1}+\alpha_{2}\leq 1.$ 
	\end{enumerate} 
\end{proposition} 
\begin{proof} 
	We know that if $\alpha\in \M_{m,n}$ and $f\in M_{n}(V^*)$ then $\Vert \alpha f\alpha^{*} \Vert_{m} \leq \Vert \alpha \Vert^{2} \Vert f\Vert_{n}$\cite{R98}. Also, if $f\in M_{n}(V^*)^{+},$ then $\alpha f\alpha^*\in M_{m}(V^*)^{+}$ \cite[Lemma 4.2]{CE77}. Using these argument, we can prove $(1)$ and $(2)$ if we choose $\alpha = \begin{bmatrix} I_{m} & 0_{m,n} \end{bmatrix}$, $\alpha = \begin{bmatrix} 0_{n,m} & I_{n} \end{bmatrix}$ and $\alpha = \begin{bmatrix} e^{i\theta}I_{m} & 0 \\ 0 & I_{n} \end{bmatrix}$ respectively. In particular, $\left\Vert \begin{bmatrix} f_{11} & \pm f_{12} \\ \pm  f_{12}^*& f_{22} \end{bmatrix} \right\Vert_{m+n} \leq 1.$ Thus as $2 \begin{bmatrix} f_{11} &  0 \\ 0 & f_{22} \end{bmatrix} = \begin{bmatrix} f_{11} & f_{12} \\ f_{12}^*& f_{22} \end{bmatrix} + \begin{bmatrix} f_{11} & -f_{12} \\ -f_{12}^*& f_{22} \end{bmatrix}$ and $2 \begin{bmatrix} 0 & f_{12} \\ f_{12}^*&0 \end{bmatrix} = \begin{bmatrix} f_{11} & f_{12} \\ f_{12}^*& f_{22} \end{bmatrix} - \begin{bmatrix} f_{11} & -f_{12} \\ -f_{12}^*& f_{22} \end{bmatrix},$ $(3)$ follows from triangle inequality. 
	
	$(4)$ As  
	$$\left\Vert \begin{bmatrix} f_{12}^* & 0 \\ 0 & f_{12} \end{bmatrix} \right\Vert_{2n} = \left\Vert \begin{bmatrix} 0 & I_n \\ I_n & 0 \end{bmatrix} \begin{bmatrix} 0 & f_{12} \\ f_{12}^{*} & 0 \end{bmatrix} \right\Vert_{2n} \le \left\Vert \begin{bmatrix} 0 & f_{12} \\ f_{12}^* & 0 \end{bmatrix} \right\Vert_{2n} \leq 1,$$
	we have  
	$$ \Vert f_{12} + f_{12}^* \Vert_n \le \Vert f_{12}^* \Vert_n + \Vert f_{12} \Vert_n \le \left\Vert \begin{bmatrix} f_{12}^* & 0 \\ 0 & f_{12} \end{bmatrix} \right\Vert_{2n} \le 1.$$
	Since $ f_{12}+f_{12}^*\in M_{n}(V^*)_{sa},$ by Lemma \ref{d4}, we may conclude that  $f_{12}+f_{12}^*\in \co(Q_{n}(V)\cup (-Q_{n}(V))).$ 
	
	$(5)$ As $f\in Q_{n}(V)\subset M_{n}(V^*)^+$, and $\gamma_{i}\in M_{n, n_{i}}$, we have $\gamma_{i}^* f \gamma_{i} \in M_{n_{i}}(V^*)^+$ for $1 \le i \le k$. Thus $\oplus_{i=1}^k \gamma_{i}^* f\gamma_{i} \in M_{\sum_{i=1}^k n_{i}}(V^*)^+$. We show that $\Vert \oplus_{i=1}^k \gamma_{i}^* f\gamma_{i} \Vert\leq 1$. Let $v \in (M_{\sum_{i=1}^k n_i}(V)_{sa})_{1}$, say $v= [v_{i,j}]$ where $v_{i,j}\in M_{n_{i}, n_{j}}(V)$ and $v_{i,j}^*= v_{j,i}, 1\leq i,j\leq k$. Then
	\begin{align*}
	\vert \langle \oplus_{i=1}^k\gamma_{i}^* f\gamma_{i},v \rangle\vert &=  \vert \sum_{i=1}^n\langle \gamma_{i}^* f\gamma_{i}, v_{ii}\rangle\vert \\
	& = \vert \sum_{i=1}^n \langle f, \gamma_{i}^{*T}v_{i,i}\gamma_{i}^{T} \rangle \vert \\
	& \leq \Vert \sum_{i=1}^k \gamma_{i}^{*T} v_{i,i}\gamma_{i}^{T}\Vert  \textrm{ for } f \in Q_{n}(V).
	\end{align*}
	Since $\sum_{i=1}^k \gamma_{i} \gamma_{i}^* \leq I_{n}$, we have 
	$$\left\Vert \sum_{i=1}^k \gamma_{i}^{*T}\gamma_{i}^T \right\Vert = \left\Vert \left( \sum_{i=1}^k \gamma_{i} \gamma_{i}^* \right)^{T} \right\Vert = \left\Vert \sum_{i=1}^k \gamma_{i} \gamma_{i}^* \right\Vert \leq 1.$$  
	Thus $\sum_{i}\gamma_{i}^{*T}\gamma_{i}^{T}\leq I_{n}$. Since $\Vert v_{i,i}\Vert_{n_{i}} \leq \Vert v\Vert_{\sum_{i=1}^k n_{i}} \leq 1$ for $1 \le i \le k$ and since $\{(M_{n}(V)_{sa})_{1}\}$ is a matrix convex set, we find that $\Vert \sum_{i=1}^k \gamma_{i}^{*T} v_{i,i} \gamma_{i}^{T} \Vert \leq 1$. Thus $\Vert \oplus_{i=1}^k \gamma_{i}^* f \gamma_{i} \Vert \leq 1$ so that $\oplus_{i=1}^k \gamma_{i}^* f \gamma_{i} \in Q_{\sum_{i=1}^k n_{i}}(V)$.
	
	$(6)$ Let $f= \begin{bmatrix} f_{11}& f_{12} \\ F_{12} & f_{22} \end{bmatrix} \in Q_{m+n}(V).$ Then by $(3)$, $f_{11} \in M_{m}(V^*)^+$ and $f_{22} \in M_{n}(V^*)^+$ and we have $\Vert f_{11} \Vert_{m} + \Vert f_{22} \Vert_{n} \leq 1$. Find $\widehat{f_{11}} \in S_{m}(V),$ $\widehat{f_{22}} \in S_{n}(V)$ such that $f_{11}=\Vert f_{11} \Vert_{m} \widehat{f_{11}}$ and $f_{22} = \Vert f_{22} \Vert_{n} \widehat{f_{22}}.$ Thus $(2)$ holds.
\end{proof}
\begin{remark} Let $V$ be a C$^{\ast}$-ordered space. Then by Proposition \ref{b4}, $\{Q_{n}(V)\}$ is an $L^{1}$-matrix convex set with $\lead(Q_{n}(V)) = S_{n}(V).$ In particular, $M_{n}(\T(H))_{1}^{+}$ is an $L^{1}$-matrix convex set. 
\end{remark}    
\section{A Quantized $A_0(K)$-space} 

Throughout in this section, we shall assume that $V$ is a $*$-locally convex space and that $\{K_{n}\}$ is an $L^{1}$- matrix convex set in $V_{sa}$. We shall also assume that $M_{n}(V)^{+} = \cup_{r=1}^{\infty } r K_{n}$ is a cone in $M_n(V)_{sa}$ for all $n$ (so that $(V, \{ M_n(V)^+\})$ is a matrix ordered space, by ${\bf L_1}$) such that $V^+$ is proper and generating.  For each $n$, we define
\begin{multline*}A_{0}(K_{n},M_{n}(V)):=\{a:K_{n}\mapsto \C | a \mbox{ is continuous and affine}; a(0)=0; \mbox{and} \\
a \mbox{ extends to a continuous linear functional} ~ \tilde{a}:M_{n}(V)\mapsto \C \}.
\end{multline*}
Let $a\in A_{0}(K_{n}, M_{n}(V))$. Since $\{ K_{n} \}$ is an $L^{1}$-matrix convex set and since $K_{n}$ spans $M_{n}(V),$ for $v\in M_{n}(V)$, we have $v= \sum_{j=1}^{r} \lambda_{j} v_{j} + i\sum_{k=1}^{r}\lambda_{k}^{\prime} v_{k}^{\prime}$ for some $v_{j},v_{j}^{\prime} \in K_{n}$ and $\lambda_{i}, \lambda_{j}^{\prime} \in \R.$ Thus $\tilde{a}(v) =\sum_{j=1}^{r} \lambda_{j}a(v_{j}) + i\sum_{k=1}^{r} \lambda_{k}^{\prime} a(v_{k}^{\prime}).$ Therefore, such an extension is always unique. For $a \in A_{0}(K_{n}, M_{n}(V)),$ we define $a^*(u)=\overline{a(u)}$ for all $u\in K_{n}$ so that $\tilde{a^*}(u)=\overline{\tilde{a}(u^{*})}$ for all $u\in M_{n}(V).$ Then $a\mapsto a^{*}$ is an involution. We set
$$A_{0}(K_{n},M_{n}(V))_{sa}=\{ a\in A_{0}(K_{n}, M_{n}(V)): a^{*}=a\}.$$ 
We consider the following algebraic operations:
\begin{enumerate}
	\item For $\alpha\in \M_{m,n}, \, \beta \in \M_{n,m} \text{ and } a\in A_{0}(K_{n}, M_{n}(V)),$ we define 
	$$\alpha a \beta(v) = \tilde{a}(\alpha^{T} v \beta^{T}) ~\textrm{for all}~ v \in K_{m}.$$ 
	Then $\alpha a \beta\in A_{0}(K_{m}, M_{m}(V))$. In fact, the map $v\mapsto \alpha^T v \beta^T$ from $M_{m}(V)$ to $M_{n}(V)$ is continuous so that the map $v \mapsto \tilde{a}(\alpha^T v\beta^T)$ from $M_{m}(V)$ into $\C$ is also continuous. Thus $\widetilde{\alpha v \beta}: M_{m}(V) \mapsto \C$ is continuous and hence
	$\alpha a\beta\in A_{0}(K_{m}, M_{m}(V)).$ 
	\item For $a\in A_{0}(K_{n}, M_{n}(V))$ and $b\in A_{0}(K_{m}, M_{m}(V)),$ we define 
	$$ (a\oplus b)(v)=a(v_{11})+b(v_{22})$$
	for all $v\in K_{n+m}$ where  
	$v=\begin{bmatrix}
	v_{11}&v_{12}\\
	v_{12}^*& v_{22}
	\end{bmatrix}$ for some $v_{11}\in K_{n}, v_{22}\in K_{m}, v_{12}\in M_{n,m}(V).$ Then 
	$a\oplus b \in A_{0}(K_{n+m},M_{n+m}(V))$. In fact, the maps $v\mapsto v_{11}$ from $K_{m+n}$ into $K_{m}$ and $v\mapsto v_{22}$ from $K_{m+n}$ into $K_{m}$ are continuous so that $v\mapsto a(v_{11})+ b(v_{22})$ is also continuous. As $\widetilde{a \oplus b} = \tilde{a} \oplus \tilde{b},$ we see that $\widetilde{a \oplus b}$ is also continuous from $M_{m+n}(V)\mapsto \C.$ Therefore, $a\oplus b \in A_{0}(K_{n+m},M_{n+m}(V)).$
\end{enumerate}
It is easy to deduce from the definition that $(\alpha a\beta)^{*}=\beta^* a^*\alpha^{*}$ and that
$(a\oplus b)^{*}=a^{*}\oplus b^{*}.$ We define
$$A_{0}(K_{n}, M_{n}(V))^{+}:=\{a\in A_{0}(K_{n},M_{n}(V))_{sa}: a(f)\geq 0 ~~\forall f\in
K_{n}\}.$$
\begin{lemma}\label{a2} 
	For $a\in  A_{0}(K_{m}, M_{m}(V))^{+}, b\in  A_{0}(K_{n}, M_{n}(V))^{+}$ and $\alpha \in \M_{m,n},$ we have 
	\begin{enumerate}
		\item $\alpha^*a\alpha\in  A_{0}(K_{n}, M_{n}(V))^{+},$ 
		\item $a\oplus b\in  A_{0}(K_{m+n}, M_{m+n}(V))^{+}.$
	\end{enumerate} 
\end{lemma} 
\begin{proof} 
	\begin{enumerate} 
		\item Let $\alpha \in \M_{m,n},$ $a \in A_{0}(K_{m}, M_{m}(V))^{+}$ and let $v \in K_{n}.$ Without any loss of generality, we may assume that $\Vert \alpha \Vert \leq 1.$ Then, by the definition of an $L^{1}$-matrix convex set, we have $\alpha^{T^*} v \alpha^{T} \in K_{m}$. Thus $\alpha^* a\alpha (v)=a(\alpha^{T^*}v\alpha^{T})\geq 0$ so that $\alpha^* a \alpha\in A_{0}(K_{n}, M_{n}(V))^{+}.$ 
		\item Let $a\in A_{0}(K_{m}, M_{m}(V))^{+},b\in A_{0}(K_{n}, M_{n}(V))^{+}$ and let $u\in K_{m+n}$ with $u = \begin{bmatrix} u_{11} &u_{12} \\ u_{12}^*& u_{22}\end{bmatrix}$ for some $u_{11} \in K_{m}, u_{22} \in K_{n}$ and $u_{12} \in M_{m,n}(V).$ Then  
		$$(a \oplus b)(u)=a(u_{11}) + b(u_{22}) \geq 0$$ 
		so that $a \oplus b \in  A_{0}(K_{m+n}, M_{m+n}(V))^{+}.$ 
	\end{enumerate}
\end{proof} 
Next, for $a \in A_{0}(K_{n}, M_{n}(V))$, we define 
$$\Vert a\Vert_{\infty,n}=\sup\left\{\left\vert \begin{bmatrix} 0 & a \\ a^* & 0 \end{bmatrix}(u) \right\vert: u \in K_{2n} \right\} ~\text{for }~ a \in A_{0}(K_{n}, M_{n}(V)).$$ 
It is routine to verify that $\Vert\cdot\Vert_{\infty, n}$ is a semi-norm on $A_{0}(K_{n}, M_{n}(V))$. We show that it is a norm. Let $a\in A_{0}(K_{n}, M_{n}(V))$ such that $\Vert a \Vert_{n} = 0.$ Let $u \in K_{n}$ and $\alpha = [\frac{1}{\sqrt2}I_{n}, \frac{1}{\sqrt2}I_{n}].$ Then $\alpha^{*} \alpha \leq I_{2n}$ and therefore, $\alpha^{*} u \alpha = \begin{bmatrix} \frac{u}{2} & \frac{u}{2} \\ \frac{u}{2} & \frac{u}{2} \end{bmatrix} \in K_{2n}.$ Also, then $\begin{bmatrix} \frac{u}{2} & i\frac{u}{2} \\ -i\frac{u}{2} & \frac{u}{2} \end{bmatrix} \in K_{2n}.$ Thus, as $\Vert a \Vert_{\infty, n}=0,$ we get 
\begin{align*}
0 = \begin{bmatrix} 0 & a \\ a^{*} & 0 \end{bmatrix}\left( \begin{bmatrix} \frac{u}{2} & i \frac{u}{2} \\ -i \frac{u}{2} & \frac{u}{2} \end{bmatrix} \right) = \tilde{a}(\frac{iu}{2}) + \tilde{a^*}(\frac{-iu}{2}) = \frac{i}{2} a(u) + \frac{-i}{2} \overline{a(u)}.
\end{align*}
Similarly,
$$
0 = \begin{bmatrix} 0 & a \\ a^{*} & 0 \end{bmatrix}\left( \begin{bmatrix} \frac{u}{2} & \frac{u}{2} \\ \frac{u}{2} & \frac{u}{2} \end{bmatrix} \right) = \frac{a(u)}{2} + \frac{\overline{a(u)}}{2}.$$
Therefore 
$a(u)\pm \overline{a(u)}=0$ for all $u\in K_{n}$ and consequently $a(u)=0$ for all $u\in K_{n}.$
Hence $a=0.$  

Further, note that $\begin{bmatrix} v_{11} & v_{12} \\ v_{12}^* & v_{22} \end{bmatrix} \in K_{2n}$ if and only if $\begin{bmatrix} v_{11} & v_{12}^* \\ v_{12} & v_{22} \end{bmatrix} \in K_{2n}$ and that 
$$\begin{bmatrix} 0 & a \\ a^{*} & 0 \end{bmatrix}\left( \begin{bmatrix} v_{11} & v_{12} \\ v_{12}^{*} & v_{22} \end{bmatrix} \right) =	\begin{bmatrix} 0 &a^* \\ a & 0	\end{bmatrix}\left( \begin{bmatrix}	v_{11} & v_{12}^* \\ v_{12} & v_{22} \end{bmatrix} \right)$$ 
for $a\in A_{0}(K_{n}, M_{n}(V))$. Thus $\Vert a^*\Vert_{\infty,n}=\Vert a\Vert_{\infty,n}$ for all $a\in A_{0}(K_{n}, M_{n}(V)).$

\begin{lemma}\label{a4}
	If $a\in A_{0}(K_{n},M_{n}(V))_{sa},$ then $$\Vert a\Vert_{\infty,n}=\sup\{\vert
	a(v)\vert: v\in K_{n}\}.$$ 
	In particular, we have 
	$$\Vert a \Vert_{\infty, n} = \left\Vert \begin{bmatrix} 0 & a\\ a^* &0 \end{bmatrix} \right\Vert_{\infty, 2 n}$$ 
	for every $a \in A_{0}(K_{n}, M_{n}(V))$.
\end{lemma} 
\begin{proof} 
	Put $r_{n}(a)=\sup\{\vert a(v)\vert: v\in K_{n}\}.$ Since $K_{2n}$ is a compact set, we have $\Vert a \Vert_{n} = \left\vert \begin{bmatrix} 0 & a \\ a & 0 \end{bmatrix}(v) \right\vert$ for some $v \in K_{2n}.$ Let $v = \begin{bmatrix} v_{11} & v_{12} \\ v_{12}^* & v_{22} \end{bmatrix}$. Since $\{ K_{n} \}$ is an $L^{1}$-matrix convex set, we have  $v_{12} + v_{12}^{*} \in \co(K_{n} \cup (-K_{n})).$ As $K_{n}$ is convex, there are $v, w \in K_{n}$ and $\lambda \in [0,1]$ such that $v_{12} + v_{12}^{*} = \lambda u - (1 - \lambda) w.$ Thus 
	\begin{align*}
	\Vert a\Vert_{\infty,n}&= \vert \tilde{a}(v_{12})+\tilde{a}(v_{12}^*)\vert=\vert \tilde{a}(v_{12}+v_{12}^*)\vert\\
	&=\vert \tilde{a}(\lambda u-(1-\lambda)w)\vert=\vert \lambda a(u)-(1-\lambda)a(w)\vert\\
	&\leq \lambda r_{n}(a)+(1-\lambda)r_{n}(a)= r_{n}(a)
	\end{align*} 
	Again as $K_{n}$ is a compact convex set, we have $r_{n}(a)=\vert a(v)\vert$ for
	some $v\in K_{n}.$ 
	Since $\{K_{n}\}$ is an $L^{1}$-matrix convex set, we have 
	$\begin{bmatrix}
	\frac{v}{2} & \frac{v}{2}\\
	\frac{v}{2}&\frac{v}{2}
	\end{bmatrix}\in K_{2n}.$
	Therefore,
	\begin{align*}
	r_{n}(a)=
	\left\vert
	\begin{bmatrix}
	0 & a\\
	a &0
	\end{bmatrix}
	\left(
	\begin{bmatrix}
	\frac{v}{2} & \frac{v}{2}\\
	\frac{v}{2}&\frac{v}{2}
	\end{bmatrix}\right)
	\right\vert \leq \Vert a\Vert_{\infty,n}.
	\end{align*}
\end{proof} 
\begin{corollary}
	For $a\leq b\leq c$ in $A_{0}(K_{n}, M_{n}(V))_{sa}$, we have 
	$$\Vert b \Vert_{\infty, n}\leq \max\{\Vert a\Vert_{\infty, n}, \Vert c\Vert_{\infty,n}\}.$$ 
\end{corollary}
\begin{proof}
	Let $a\leq b\leq c$ in $A_{0}(K_{n}, M_{n}(V))_{sa}$. Then $a(u)\leq b(u)\leq c(u)$ for all $u\in K_{n}$ so that $\vert b(u)\vert\leq \max \{\vert a(u)\vert, \vert c(u)\vert \}.$ Thus by Lemma \ref{a4}, we get $\vert b(u)\vert\leq \max\{\Vert a\Vert_{\infty,n},\Vert b\Vert_{\infty,n} \}$ for all $u\in K_{n}$ so that $\Vert b\Vert_{\infty, n}\leq \max\{\Vert a\Vert_{\infty, n}, \Vert c\Vert_{\infty,n}\}.$
\end{proof} 

\begin{proposition}\label{a7} 
	Let $\{K_{n}\}$ be an $L^{1}$-matrix convex set in $V.$ Then $\{ \Vert\cdot\Vert_{\infty, n}\}$ satisfies the following conditions: 
	\begin{enumerate} 
		\item $\Vert a\oplus b\Vert_{\infty,m+n}=\max \{\Vert a\Vert_{\infty,m}, \Vert b\Vert_{\infty,n}\}$ for all $ a\in A_{0}(K_{m}, M_{m}(V))$ and $b\in A_{0}(K_{n}, M_{n}(V));$ 		
		\item $\Vert \alpha a \beta \Vert_{\infty,m} \leq \Vert \alpha \Vert \Vert a  \Vert_{\infty,n} \Vert \beta \Vert $ for all $a \in A_{0}(K_{n}, M_{n}(V)),\alpha \in \M_{m,n} ~\mbox{ and}~ \beta \in \M_{n,m}.$ 
	\end{enumerate} 
\end{proposition} 

\begin{proof} 
	We shall prove this result in several steps.
	
	\vskip 2mm
	{\bf Step I.} $\Vert a\oplus b\Vert_{\infty,m+n}=\max \{\Vert a\Vert_{\infty,m}, \Vert b \Vert_{\infty,n}\}$ for all $ a\in A_{0}(K_{m}, M_{m}(V))_{sa}$ and $b \in
	A_{0}(K_{n}, M_{n}(V))_{sa}$. 
	
	\vskip 2mm 
	Let $a\in A_{0}(K_{m}, M_{m}(V))_{sa}$ and $b\in A _{0}(K_{n}, M_{n}(V))_{sa}.$ Now for every $v\in K_{m},$ we have 
	$$\vert a(v)\vert=\vert (a\oplus b)(v\oplus 0)\vert.$$ 
	Since $\{K_{n}\}$ is an $L^{1}$-matrix convex set, we have $v\oplus 0\in 
	K_{m+n}$ whenever $v\in K_{m}.$  Therefore from Proposition \ref{a4}, we conclude that  $\Vert a \Vert_{\infty,m}\leq  \Vert a\oplus b\Vert_{\infty,m+n}.$ Similarly, we can show that  $\Vert b \Vert_{\infty,n}\leq \Vert a\oplus b\Vert_{\infty,m+n}.$ 
	
	Conversely, let $v=\begin{bmatrix} v_{11}& v_{12}\\ v_{12}^* &v_{22} \end{bmatrix}\in K_{m+n}.$ Then there exist $\widehat{v_{11}}\in \lead(K_{m}), \widehat{v_{22}}\in \lead(K_{n})$ and $\alpha_{1},\alpha_{2}\in [0,1]$ with $\alpha_{1}+\alpha_{2}\leq 1$ such that $v_{11}=\alpha_{1}\widehat{v_{11}}, v_{22}=\alpha_{2}\widehat{v_{22}}.$ Thus 
	\begin{align*} 
	\vert (a\oplus b)(v)\vert&=\vert a( v_{11})+ b(v_{22})\vert\\ &=\vert\alpha_{1}a(\widehat{v_{11}})+\alpha_{2}b(\widehat{v_{22}})\vert\\ 
	&\leq \alpha_{1} \Vert a \Vert_{\infty, m}+ \alpha_{2}\Vert b\Vert_{\infty,n}\\ 
	&\leq \max\{\Vert a\Vert_{\infty,m}, \Vert b\Vert_{\infty,n}\}. 
	\end{align*} 
	Therefore $\Vert a\oplus b\Vert_{\infty,m+n}=\max\{\Vert a\Vert_{\infty,m} 
	\Vert b\Vert_{\infty,n}\}.$
	
	\vskip 2mm
	{\bf Step II.} For $a\in A_{0}(K_{n}, M_{n}(V))_{sa}$ and $\alpha \in \M_{m,n}$, we have 
	$\Vert \alpha^* a \alpha \Vert_{\infty,n} \leq \Vert \alpha \Vert^{2} \Vert a \Vert_{\infty,n}.$
	
	\vskip 2mm
	Let $a\in A_{0}(K_{m}, M_{m}(V))_{sa}$ and $\alpha\in \M_{m,n}$ such that $\Vert
	\alpha\Vert\leq 1$ and let $v\in K_{n}.$ Since $\{K_{n}\}$ is an $L^{1}$-matrix convex set and $\alpha^{*T}\alpha^{T}\leq I_{m},$ 
	we have 
	$\alpha^{T^{*}} v\alpha^T\in K_{m}.$ Also we know that
	\begin{align*}
	\vert (\alpha^{*} a\alpha)(v)\vert&=\vert a(\alpha^{T^*}v\alpha^{T})\vert.
	\end{align*}
	Since $a$ is self-adjoint, by Proposition \ref{a4},
	we have $\Vert\alpha^{*} a\alpha\Vert_{\infty,n}\leq\Vert a\Vert_{\infty,n}$ for $a=a^{*}.$ 
	In particular, if $m = n$ and if $\alpha \in \M_m$ is unitary, then $\Vert \alpha^{*} a \alpha
	\Vert_{\infty,m} = \Vert a \Vert_{\infty,m}.$
	Also, in general, for $a\in A_{0}(K_{n}, M_{n}(V))_{sa}$ and $\alpha \in \M_{m,n}$, we have 
	$$\Vert\alpha^*a\alpha\Vert_{\infty,n}\leq \Vert\alpha\Vert^{2}\Vert
	a\Vert_{\infty,n}.$$
	
	Now we are ready to prove $(1)$ and $(2).$
	\begin{enumerate}
		\item Let $a\in A_{0}(K_{m}, M_{m}(V))$ and $b\in A_{0}(K_{n}, M_{n}(V)).$ Put $\gamma =  \begin{bmatrix}
		I_{m}   &0     &0      &0    \\
		0       &0    &I_{n} &0       \\
		0       &I_{m}&0      &0        \\
		0       &0     &0       &I_{n}
		\end{bmatrix}$. Then $\gamma \in M_{2m + 2n}$ is a unitary and 
		\begin{align*}
		\gamma^* 
		\begin{bmatrix}
		0 & a\oplus b\\
		a^* \oplus b^* &0
		\end{bmatrix} \gamma =  
		\begin{bmatrix}
		0 & a  \\
		a^{*} & 0 \end{bmatrix} \oplus \begin{bmatrix} 0 & b \\ b^* & 0 \end{bmatrix}
		\end{align*} 
		so that 
		\begin{align*}
		\left\Vert \begin{bmatrix}
		0 & a\oplus b\\
		a^* \oplus b^* &0
		\end{bmatrix}  \right\Vert_{\infty,2(m+n)} 
		= \left\Vert \begin{bmatrix}
		0 & a  \\
		a^{*} & 0 \end{bmatrix} \oplus \begin{bmatrix} 0 & b \\ b^* & 0 \end{bmatrix} \right\Vert_{\infty,2m+2n}
		\end{align*} 
		by Step II. Thus by Lemma \ref{a4}, we have
		\begin{align*}
		\Vert a\oplus b\Vert_{m+n} 
		&= 
		\left \Vert
		\begin{bmatrix}
		0               & a\oplus b\\
		a^*\oplus b^* &0
		\end{bmatrix}\right\Vert_{\infty,2(m+n)} \\
		&= \left\Vert \begin{bmatrix}
		0 & a  \\
		a^{*} & 0 \end{bmatrix} \oplus \begin{bmatrix} 0 & b \\ b^* & 0 \end{bmatrix} \right\Vert_{\infty,2m+2n} \\
		&= \max\left\{
		\left\Vert \begin{bmatrix}
		0 &a\\
		a^* &0
		\end{bmatrix}
		\right\Vert_{\infty,2m}, 
		\left\Vert\begin{bmatrix}
		0    &b\\
		b^* &0
		\end{bmatrix}
		\right\Vert_{\infty,2n}
		\right\} \notag \\ 
		&= 
		\max \{ \Vert a \Vert_{\infty,m}, \Vert b \Vert_{\infty,n} \}.
		\end{align*} 
		\item Let $\alpha\in \M_{m,n}, a\in A_{0}(K_{n}, M_{n}(V)) \mbox{ and }\beta\in \M_{n,m}.$
		Then by Lemma \ref{a4}, we have 
		\begin{align*}
		\Vert \alpha a\beta\Vert_{\infty,m}=&
		\left\vert\left\vert
		\begin{bmatrix}
		0              &\alpha a\beta\\
		\beta^*a^*\alpha^* &          0
		\end{bmatrix}
		\right\vert\right\vert_{\infty,2m}
		\end{align*}
		For $t\in \R^+\setminus \{0\},$ we have 
		\begin{align*}
		&\begin{bmatrix}
		t \alpha      &           0            \\
		0        &  \frac{1}{t}\beta^*
		\end{bmatrix}
		\begin{bmatrix}
		0           &a\\
		a^*          &   0
		\end{bmatrix}
		\begin{bmatrix}
		t\alpha^*            &0                    \\
		0              &\frac{1}{t}\beta
		\end{bmatrix}
		=
		\begin{bmatrix}
		0              &\alpha a\beta\\
		\beta^* a^*\alpha^* &          0
		\end{bmatrix}.
		\end{align*}
		
		Thus,
		\begin{align*}
		\Vert\alpha a \beta \Vert_{\infty,m}
		&\leq
		\left\Vert
		\begin{bmatrix}
		t \alpha      &           0            \\
		0       &  \frac{1}{t}\beta^*
		\end{bmatrix}
		\right\Vert
		\left\Vert
		\begin{bmatrix}
		0           & a\\
		a^*          &   0
		\end{bmatrix}
		\right\Vert_{\infty,2n}
		\left\Vert
		\begin{bmatrix}
		t\alpha            &0                    \\
		0              &\frac{1}{t}\beta
		\end{bmatrix}
		\right\Vert \\
		&\leq \max\{\Vert t^{2}\alpha\alpha^*\Vert,\Vert
		\frac{1}{t^2}^{2}\alpha\alpha^*\Vert \}\Vert a \Vert_{\infty,n} \\
		&\leq \max\{t^2 \Vert\alpha\Vert^2,\frac{1}{t^2} \Vert\beta\Vert^2\} \Vert a \Vert_{\infty,n}.
		\end{align*}
		Taking infimum over $t\in \R^{+}\setminus\{0\},$ we may conclude that
		$\Vert\alpha a\beta\Vert_{\infty,m}\leq \Vert\alpha\Vert \Vert
		a\Vert_{\infty,n}\Vert\beta\Vert.$ 
	\end{enumerate}
\end{proof} 
Finally, for each $n \in \N$. we define $\Phi_n: M_{n}(A_{0}(K_{1}, V)) \to A_{0}(K_{n},M_{n}(V))$ as follows: Let $a_{ij} \in A_{0}(K_{1}, V)$ for $1 \le i, j \le n$. Define 
$$\Phi_n([a_{ij}]): K_n \to \C ~ \textrm{given by} ~ \Phi_n([a_{ij}])([v_{ij}]) = \sum_{i, j = 1}^n \widetilde{a_{ij}}(v_{ij}) ~ \textrm{for all} ~ [v_{ij}] \in K_n.$$ 
Now, it is routine to show that $\Phi_n([a_{ij}]) \in A_{0}(K_{n},M_{n}(V))$. (Note that $\Phi_n$ is an amplification of $\Phi_1$. That is, $\Phi_n([a_{ij}]) = [\Phi_1(a_{ij})]$, if $[a_{ij}] \in M_n(A_{0}(K_{1}, V))$.) Under this identification, we note that $[a_{i,j}]^* = [a_{j,i}^*]$ is an involution in $M_{n}(A_{0}(K_{1}, V))$ so that $\Phi_1$ is a $*$-isomorphism. 

For each $n \in\N$, we set 
$$M_{n}(A_{0}(K_{1}, V))^+ := \left\{ [a_{ij}] \in M_{n}(A_{0}(K_{1}, V))_{sa}: \sum_{i,j=1}^{n} \widetilde{a_{i,j}}(v_{i,j})\geq 0 ~\textrm{for all} ~ [v_{i,j}] \in K_{n} \right\}$$ 
and transport the norm 
$$  \Vert [a_{i,j}]\Vert_{n} := \Vert
\Phi_{n}([a_{i,j}])
\Vert_{\infty,n}$$
for all $[a_{i,j}]\in M_{n}(A_{0}(K_{1}, V)).$ Now, the next result is an assimilation of the observations made in this section.
\begin{theorem}\label{a9} 
	$\left( A_{0}(K_{1}, V), \{ M_{n}(A_{0}(K_{1}, V))^+ \}, \{ \Vert\cdot\Vert_n \} \right)$ is a C$^*$-ordered operator space. 
\end{theorem} 
\begin{remark} 
	Let $\{K_{n}\}$ be an $L^{1}$-matrix convex set of $V$. Then by \cite[Theorem 1.7]{KARN11} that there is a complete order isometry $\phi:A_{0}(K_{1},V)\mapsto \A$ for some C$^{\ast}$-algebra $\A.$     
\end{remark}

\section{Completely Regularity}
In this section, we shall find conditions on an $L^{1}$-matrix convex set $\{ K_{n} \}$ in a given $*$-locally convex space $V$ such that $\left( A_{0}(K_{1}, V), \{ M_{n}(A_{0}(K_{1}, V))^+ \}, \{ \Vert\cdot\Vert_n \} \right)$ becomes a matrix order unit space. Let $E$ be a real locally convex space, and $M$ be a compact convex set in $E.$ Then $M$ is said to be \emph{regularly embedded} in $E$ if the following conditions hold:
\begin{enumerate}
	\item $M$ spans $E;$
	\item there exists a hyperplane $H$ containing $M$ such that $0\notin H;$
	\item canonical embedding $x\mapsto \chi(x),$ mapping $E$ to $A(M)^*_{w^*}$ is a topological isomorphism. \cite[ Chapter II.2]{ALF71}
\end{enumerate}
We propose a matricial version of regular embedding of an $L^{1}$-matrix convex set $\{ K_{n} \}$ in a given $*$-locally convex space $V$. Let $L_n$ be the lead of $K_n$ for each $n$. We shall call $\{ L_n \}$ the \emph{matricial lead} of $\{ K_{n} \}$. We also assume that $M_{n}(V)^{+}= \cup_{r=1}^{\infty }r K_{n}$ is a cone in $M_n(V)_{sa}$ for all $n$ (so that $(V, \{ M_n(V)^+\})$ is a matrix ordered space, by ${\bf L_1}$) such that $V^+$ is proper and generating. First, we consider the following notion.

\begin{definition}
	Let $\{K_{n}\}$ be an $L^{1}$-matrix convex set with its matricial lead $\{L_{n}\}$. We shall call $\{ K_n \}$ an \emph{$L^1$-matricial cap} of $V$ if
	\begin{enumerate}
		\item $L_{1}$ is convex; and 
		\item  if $v\in L_{m+n}$ with 
		$v=\begin{bmatrix}
		v_{11}& v_{12} \\
		v_{12}^*& v_{22} 
		\end{bmatrix}$ 
		for some $v_{11}\in K_{m}, v_{22}\in K_{n}$ and $v_{12}\in M_{m,n}(V)$ so that
		$v_{11}=\alpha_{1}\widehat{v_{1}}, v_{22}=\alpha_{2}\widehat{v_{2}}$ for some
		$\widehat{v_{1}}\in L_{m}, \widehat{v_{2}}\in L_{n}$ and $\alpha_{1},\alpha_{2}\in [0,1],$ then
		$\alpha_{1}+\alpha_{2}= 1.$
	\end{enumerate}
\end{definition} 
\begin{theorem}
	Let $\{ K_n \}$ be an $L^1$-matricial cap of $V$. Then $L_n$ is convex for every $n$.
\end{theorem}
\begin{proof}
	We shall prove this result in several steps. 
	
	{\bf Step I. $L_{2}$ is convex.}
	
	Let $v= \begin{bmatrix} v_{11} & v_{12} \\ v_{12}^*& v_{22} \end{bmatrix}, w= \begin{bmatrix} w_{11} & w_{12} \\ w_{12}^*& w_{22} \end{bmatrix}\in L_{2}$ and let $\lambda \in [0,1]$. Then by $(2)$, we have  $v_{11}= \alpha_{1} \widehat{v_{1}}, v_{22} = \alpha_{2} \widehat{v_{2}}, \alpha_{1} + \alpha_{2} = 1$, for some $\widehat{v_{1}}, \widehat{v_{2}} \in L_{1}$, and $w_{11}= \beta_{1} \widehat{w_{1}}, w_{22} = \beta_{2} \widehat{w_{2}}, \beta_{1} + \beta_{2} = 1$, for some $\widehat{w_{1}}, \widehat{w_{2}} \in L_{1}$. Now 
	$$u := \lambda v + (1 - \lambda) w = \begin{bmatrix} \lambda v_{11} + (1-\lambda)w_{11} & \lambda v_{12} + (1-\lambda)w_{12} \\ \lambda v_{12}^* + (1-\lambda)w_{12}^* & \lambda v_{22} + (1-\lambda)w_{22} \end{bmatrix}\in K_{2}.$$ 
	Let $u= \begin{bmatrix}	u_{11} & u_{12} \\ u_{12}^* & u_{22} \end{bmatrix}$ so that $u_{11}=  \lambda v_{11} + (1-\lambda)w_{11} = \lambda \alpha_{1}\widehat{ v_{1}}+ (1-\lambda)\beta_{1}\widehat{w_{1}}$ and
	$u_{22}=  \lambda v_{22}+ (1-\lambda)w_{22} = \lambda \alpha_{2}\widehat{ v_{2}}+ (1-\lambda)\beta_{2}\widehat{w_{2}}$. %Then $u_{11}, u_{22}\in  K_{1}$.  
	Since $L_{1}$ is convex,
	$\widehat{u_{1}}=(\lambda \alpha_{1}+ (1-\lambda)\beta_{1})^{-1}u_{11}\in L_{1}$ and $\widehat{u_{2}}=(\lambda \alpha_{2}+ (1-\lambda)\beta_{2})^{-1}u_{22}\in L_{1}$. Put $(\lambda \alpha_{1}+ (1-\lambda)\beta_{1})
	=\gamma_{1}$ and $(\lambda \alpha_{2}+ (1-\lambda)\beta_{2})=\gamma_{2}$, then $u= \begin{bmatrix}
	\gamma_{1}\widehat{u_{1}}& u_{12}\\
	u_{12}^* &\gamma_{2}\widehat{ u_{2}}
	\end{bmatrix}$ and  $\gamma_{1}+\gamma_{2}=\lambda(\alpha_{1}+\alpha_{2})+(1-\lambda)(\beta_{1}+\beta_{2})=\lambda+(1-\lambda)=1$. Let $u=\gamma \widehat{u},$ where $\widehat{u}\in L_{2}$ and $\gamma\in [0,1]$. We show that $\gamma=1$. Let $\widehat{u}= \begin{bmatrix} x_{11} & x_{12} \\ x_{12}^* & x_{22} \end{bmatrix}$. Then $x_{11}, x_{22}\in K_{1}$ with $\gamma x_{11}= u_{11}, \gamma x_{22}= u_{22}$. Thus $x_{11} = \gamma^{-1} \gamma_1 \widehat{u_1}$ and $x_{22} = \gamma^{-1} \gamma_2 \widehat{u_2}$. Since $\{ K_n\}$ is an $L^1$-matricial cap, we get $1 = \gamma^{-1} \gamma_1 + \gamma^{-1} \gamma_2 = \gamma^{-1}$. Thus  $\gamma = 1$ and consequently, $u \in L_2$. Hence $L_2$ is convex.
	
	{\bf Now, by induction, $L_{2^n}$ is convex for every $n$.}
	
	{\bf Step II. For $m, n \in \N$, we have $L_{m}$ is convex if $L_{m+n}$ is convex.}
	
	First, we show that  $v\mapsto v\oplus 0$ maps $L_{m}$ into $L_{m+n}$. Let $v \in L_{m}$. Then $v\oplus 0\in K_{m+n}$ so that $v\oplus 0= \alpha \widehat{w}$ for some $\widehat{w}\in L_{m+n}$ and $\alpha\in [0,1]$. Thus 
	$$ v = \begin{bmatrix} I_{n}& 0_{n,m} \end{bmatrix}(v\oplus 0) \begin{bmatrix} I_{n}\\0_{m,n} \end{bmatrix} = \alpha  \begin{bmatrix} I_{n}& 0_{n,m} \end{bmatrix}\widehat{w} \begin{bmatrix} I_{n}\\0_{m,n} \end{bmatrix} =\alpha w_{1}$$ 
	where $w_{1}= \begin{bmatrix} I_{n}& 0_{n,m} \end{bmatrix}\widehat{w} \begin{bmatrix} I_{n}\\0_{m,n} \end{bmatrix}\in K_{m}.$ Now, as $L_2$ is the lead of $K_2$, we have $\alpha =1$ and $w_{1}= v$. Thus $v\oplus 0= \widehat{w}\in L_{m+n}$. 
	
	Now assume that $L_{m+n}$ is convex. Let $v, w\in L_{m}$ and $\alpha \in (0,1)$. Then 
	$$(\alpha v+ (1-\alpha) w)\oplus 0 = \alpha(v\oplus 0)+(1-\alpha)(w\oplus 0)  \in L_{m+n}.$$
	Put  $u=\alpha v\oplus (1-\alpha)w$.
	Then $u\in K_{m}$ so that $u=\lambda \widehat{u}$ for some $\widehat{u}\in L_{m}$ and $\lambda \in [0,1]$. As $ \widehat{u}\in L_{m}$, we get that $\widehat{u}\oplus{0}\in L_{m+n}$.
	Now $\lambda(\widehat{u}\oplus {0})= u\oplus 0\in L_{m+n}$ so that $\lambda =1$  and $u=\widehat{u}\in L_{m}$. Thus $L_{m}$ is convex. 
	
	Hence, by Step I, $L_n$ is convex for every $n$.
\end{proof} 

When $L_{1}$ is compact and convex, we denote by $A(L_{1})$ the set of all complex valued affine functions on $L_1$. Then $A(L_{1})_{sa}$ is an order unit space so that  
%the set of real valued continuous affine function on $L_{1}.$ Therefore 
$A(L_{1})_{sa}^{*},$ the ordered Banach dual of $A(L_{1})_{sa},$ is a base normed space \cite{JEM70, ALF71}. 
%In such case,  $A((L_{1})_{sa}^{*})_{w^*}$ is stand for locally convex space, the ordered Banach dual of $A(L_{1})_{sa}^*$ with respect to $w^*$-topology.

\begin{definition}
	Let $\{K_{n}\}$ be an $L^{1}$-matrix convex set in a $*$-locally convex space $V.$ Then $\{K_{n}\}$ is called \emph{regularly embedded} in $V$ if $L_1$ is regularly embedded in $V_{sa}$. In other words, 
	\begin{enumerate}
		\item $L_{1}$ is compact and convex; and 
		\item $\chi :V_{sa}\mapsto (A(L_{1})_{sa}^{*})_{w*}$
		is an linear homeomorphism.
	\end{enumerate}
	Here $\chi(w)(a)=\lambda a(u)-\mu a(v)$ for all  for all $a\in A(L_{1})_{sa}$ if $w = \lambda u - \mu v$ for some 
	$u, v \in L_1$ and $\lambda, \mu \in \R^+$. 
\end{definition}
We note that $\chi(w)$ is well defined. To see this, let $ w = \lambda_{1} u_{1}-\mu_{1} v_{1}=\lambda_{2} u_{2}-\mu_{2} v_{2}$ for some $u_{i}, v_{i}\in
L_{1}$ and $\lambda_{i}, \mu_{i}\in \R^{+}$ for $i=1,2.$ As $L_{1}$ is convex and 
$\frac{\lambda_{1}+\mu_{2}}{\lambda_{2}+\mu_{1}}\left(\frac{\lambda_{1}u_{1}+\mu_{2}v_{1} }{\lambda_{1}+
	\mu_{2}}\right)=\frac{\lambda_{2}u_{2}+\mu_{1}v_{1} }{\lambda_{2}+ \mu_{1}}$, by Proposition \ref{c4}, we
have $\lambda_{1}+ \mu_{2}= \lambda_{2}+ \mu_{1}.$ So if $a$ is an affine function on $L_{1},$ then
$\frac{\lambda_{1}a(u_{1})+\mu_{2}a(v_{2})}{\lambda_{1}+ \mu_{2}}=a(\frac{\lambda_{1}u_{1}+\mu_{2}v_{2}
}{\lambda_{1}+ \mu_{2}})=a(\frac{\lambda_{2}u_{2}+\mu_{1}v_{1} }{\lambda_{2}+
	\mu_{1}})=\frac{\lambda_{2}a(u_{2})+\mu_{1}a(v_{1})}{\lambda_{2}+ \mu_{1}}.$ Thus
$\lambda_{1}a(u_{1})-\mu_{1}a(v_{1})=\lambda_{2}a(u_{2})-\mu_{2}a(v_{2})$ so that
$\chi(w)$ is well defined linear functional on $A(L_{1})_{sa}$ for all $u,
v\in L_{n}$ and $\lambda , \mu\in \R^+$. % such that 

\begin{theorem} \label{mou}
	Let $\{ K_{n} \}$ be a regularly embedded, $L^{1}$-matricial cap in $V$. Then $A_{0}(K_{1}, V)$ has an order unit, say $e$ so that $(A_{0}(K_{1}, V), e)\}$ is a matrix order unit space. 
\end{theorem} 
\begin{proof}
	As $L_{1}$ is the lead of $K_{1}$, there exists a mapping $e: K_{1}\setminus \{0\}\mapsto (0,1]$ given by $e(k)=\alpha$ if $k=\alpha\widehat{k}$ for some $\widehat{k}\in L_{1}$ and $\alpha\in (0,1]$. Since $\alpha$ and $\widehat{k}$ are uniquely determined by $k\in K_{1}\setminus\{0\}$, e is well defined. We extend $e$ to $K$ by putting $e(0)=0$. Since $L_{1}$ is convex, we may conclude that $e: K_{1}\mapsto [0,1]$ is affine. Again since $K_{1}$ spans $V$, we can extend $e$ to a self-adjoint linear functional  $\tilde{e}: V\mapsto \C$. Following this way, for each $n \in \N$, we can construct a self-adjoint linear functional $\widetilde{e_n}: M_n(V) \mapsto \C$ such that $\widetilde{e_n}(v) = 1$ for all $v \in L_n$. (We write $e_n$ for $\widetilde{e_n}|_{L_n}$.)
	
	We show that $\tilde{e}$ is continuous. It is suffices to show that $\tilde{e}|_{V_{sa}}$ is continuous at $0$. Let $\{ \lambda_{\alpha}u_{\alpha}-\mu_{\alpha} v_{\alpha} \}$ be a net in $V_{sa}$ for some $u_{\alpha},v_{\alpha} \in L_{1}$ and $\lambda_{\alpha},\mu_{\alpha}\in \R^+$ such that $\lambda_{\alpha}u_{\alpha}- \mu_{\alpha}v_{\alpha}\to 0.$ Since $\{K_{n}\}$ is $L^1$-regularly embedded in $V$, we get $\chi(\lambda_{\alpha}u_{\alpha}- \mu_{\alpha}v_{\alpha}) \to 0$ in $(A(L_{1})_{sa}^{*})_{w*}.$ Let $I_{L_{1}}$ be the constant map on $L_{1}$ such that $I_{L_{1}}(v)=1$ for all $v\in L_{1}.$ Then $I_{L_{1}} \in A(L_{1})_{sa}$. Thus $\chi(\lambda_{\alpha}u_{\alpha}- \mu_{\alpha}v_{\alpha})(I_{L_{1}}) \to 0$ so that $\tilde e(\lambda_{\alpha}u_{\alpha}-\mu_{\alpha}v_{\alpha})\to 0$. Now it follows that $e \in A_{0}(K_{1}, V).$ 
	
	Next, fix $n \in \N$ and consider $e^n \in M_{n}(A_{0}(K_{1}, V))$ so that by Theorem \ref{a9}, $e_0^n := \Phi_{n}\left( e^n \right) \in A_0(K_n, M_n(V)).$ We show that $e_0^n = e_n$. Let $v = [v_{i,j}]\in L_{n}$ so that $v_{i,i}\in K_{1}$ for $i = 1, \dots, n.$ Let $v_{ii} = \alpha_{i} \widehat{v_{_{i}}}$ for some $\alpha_{i}\in [0,1]$ and $\widehat{v_{i}}\in L_{n}.$ Since $\{K_{n}\}$ is $L^{1}$-matricial cap, we have $\sum_{i=1}^{n}\alpha_{i}=1.$ Thus 
	\begin{align*}
	e_0^n(v) = \sum_{i=1}^{n} e(v_{i,i}) = \sum_{i=1}^{n} \alpha_{i} e(\widehat{v_{i}}) =\sum_{i=1}^{n} \alpha_{i} = 1
	\end{align*}
	so that $e_0^n(v)=e_{n}(v)$ for all $v\in L_{n}$. Since $L_n$ is the lead of $K_n$ and since $K_n$ spans $M_n(V)$, it follows that $e_{n} = \Phi_{n}\left( e^n \right)$ for all $n \in \N$. 
	%and that $e_{n} \in A_{0}(K_{n}, M_{n}(V)).$
	
	Note that $\Vert e \Vert_{\infty, 1} = 1$. We show that $e$ is an order unit for $A_{0}(K_{1}, V)_{sa}$. 
	To see this, let $a\in A_{0}(K_{1}, V)_{sa}$. Then $\vert a(k)\vert\leq \Vert a\Vert_{\infty, 1}$ for all $k\in K_{1}$. Let $k\in K_{1}$. If $k=0$, then $a(0)=0$ so that 
	$$-\Vert a\Vert_{\infty, 1}e(0)=0=\Vert a\Vert_{\infty, 1}e(0).$$ 
	Let $k\neq 0$, then there exist a unique $\widehat{k}\in L_{1}$ and $\alpha\in (0,1]$ such that $k=\alpha \widehat{k}$. Now %$a(k)=\alpha a(\widehat{k})$ so that 
	$$- \Vert a \Vert_{\infty, 1} e(\hat{k}) = - \Vert a \Vert_{\infty, 1} \le  a(\hat{k}) \le \Vert a \Vert_{\infty, 1} = \Vert a \Vert_{\infty, 1} e(\hat{k}).$$ 
	so that
	$$
	- \Vert a\Vert_{\infty, 1} e(k)\leq a(k)\leq \Vert a\Vert_{\infty, 1}e(k)  
	$$
	for all $k\in K$.
	Thus we have  $-\Vert a\Vert_{\infty, 1}e \leq a \leq \Vert a \Vert_{\infty,1}e$
	for all $a\in A_{0}(K_{1}, V)_{sa}$. In other words, $e$ is an order unit for $A_{0}(K_{1}, V)_{sa}$ which determines $\Vert\cdot\Vert_{\infty, 1}$ as an order unit norm on it. 	Similarly, we can show that for each $n \in \N$, $e_n$ is an order unit for $A_{0}(K_{n}, M_n(V))_{sa}$ which determines $\Vert\cdot\Vert_{\infty, n}$ as an order unit norm on it. 	Again, being function space, $A_{0}(K_{n}, M_n(V))$ is Archimedean for every $n$. Hence $(A_{0}(K_{1}, V), e)$ is a matrix order unit space.
\end{proof} 
Next, we prove the completeness of $(A_{0}(K_{1}, V), e)$. 
\begin{proposition}\label{c1} 
	Let $\{K_{n}\}$ be an $L^{1}$-matrix convex set in a $*$-locally convex space $V.$ Then $\overline{A_{0}(K_{n}, M_{n}(V))_{sa}}= A_{0}(K_{n})_{sa}$.
\end{proposition}
\begin{proof} 
	By the definition, $A_{0}(K_{n}, M_{n}(V))_{sa} \subset A_{0}(K_{n})_{sa}$. Also, since $A_{0}(K_{n})_{sa}$ is norm complete, we get $\overline{A_{0}(K_{n}, M_{n}(V))_{sa}} \subset A_{0}(K_{n})_{sa}$. Conversely, let $a \in A_{0}(K_{n})_{sa}$ and $\epsilon >0$. Then $G_{K_{n}}(a)$ and $G_{K_{n}}(a+\epsilon)$ are compact convex set in $M_{n}(V)_{sa} \times \R$. Here 
	$$G_{K_n}(b + \lambda) := \{ (k, b(k) + \lambda ): k \in K_n \}$$ 
	for $b \in A_{0}(K_{n})_{sa}$ and $\lambda \in [0, \infty )$. 
	Thus $G_{K_n}(a) \cap G_{K_n}(a + \epsilon) = \emptyset$. Therefore, by the Hahn Banach separation theorem, there are $f\in (M_{n}(V)_{sa})^* (= (M_n(V)^*)_{sa})$ and $\lambda \in \R$ such that
	$$(f, \lambda)(u, a(u))< (f,\lambda)(v, a(v)+\epsilon) ~\forall u,v\in K_{n}.$$ 
	Simplifying this, we get
	$$f(u)+ \lambda a(u)< f(v)+ \lambda (a(v)+\epsilon) ~\forall u,v\in K_{n}.$$ 
	In particular, when $u= v= 0$, we get $\lambda > 0$. Similarly, for $u=0$ and $v=0$ separately, we have
	$$\lambda^{-1} f(u)+ a(u)< \epsilon \textrm{ and } \lambda^{-1}  f(v)+ a(v) >-\epsilon ~\forall u,v\in K_{n}. $$
	Let us put $a_{1}= - \lambda^{-1} f$, then $a_{1}\in A_{0}(K_{n}, M_{n}(V))_{sa}$ and $\vert a_{1}(u)- a(u)\vert<\epsilon $ for all $u\in K_{n}.$ Thus by Lemma \ref{a4}, we have $\Vert a_{1}- a\Vert_{\infty, n} \leq \epsilon$. This completes the proof.
\end{proof}
\begin{proposition} Under the assumptions of Theorem \ref{mou}, $A_{0}(K_{n}, M_{n}(V))=A_{0}(K_{n}).$
\end{proposition}
\begin{proof}
	We know that $A_{0}(K_{1}, V)\subseteq A_{0}(K_{1}).$ Let $a\in A_{0}(K_{1})$ so that $a=a_{1}+ia_{2}$ for some $a_{1}, a_{2}\in A_{0}(K_{1})_{sa}$ and let $\{\lambda_{\alpha} u_{\alpha}- \mu_{\alpha}v_{\alpha}\}$ be
	a net in $V_{sa}$ for some $u_{\alpha}, v_{\alpha}\in L_{1}$ and $\lambda_{\alpha}, \mu_{\alpha}\geq
	0$ such that $\lambda_{\alpha} u_{\alpha}- \mu_{\alpha}v_{\alpha}\to 0$. Since $K_{1}$ spans $V,$ $a_i$ has a unique linear extension $\widetilde{a_{i}}$ for $i=1,2.$ Since $\{ K_{n} \}$ is
	$L^1$-regularly embedded in $V,$ $\chi(\lambda_{\alpha} u_{\alpha}-\mu_{\alpha}v_{\alpha})\to 0$ in $(A(L_{1})_{sa}^{*})_{w*}.$ Thus 
	\begin{align*}
	\widetilde{a_{i}}(\lambda_{\alpha}u_{\alpha}-\mu_{\alpha}v_{\alpha})
	&=\lambda_{\alpha} a_{i}(u_{\alpha})-\mu_{\alpha}a_{i}(v_{\alpha})\\
	&=\lambda_{\alpha} a_{i}|_{L_{1}}(u_{\alpha})-\mu_{\alpha}a_{i}|_{L_{1}}(v_{\alpha})\\
	&=\chi(\lambda_{\alpha} u_{\alpha}-\mu_{\alpha}v_{\alpha})(a_{i}|_{L_{1}})\to 0
	\end{align*}
	Let $\widetilde{a}=\widetilde{a_{1}}+ i\widetilde{a_{2}}.$ Then $\widetilde{a}|_{K_{1}}=a$ and
	$\widetilde{a}(\lambda_{\alpha}u_{\alpha}-\mu_{\alpha}v_{\alpha})\to 0.$ Thus $\widetilde{a}$ is
	continuous on $V$ and consequently, $a\in A_{0}(K_{1}, V).$ Therefore we have $A_{0}(K_{1})=A_{0}(K_{1}, V).$
	It follows that $A_{0}(K_{1}, V)$ is $\Vert\cdot\Vert_{1}$-complete so that $(A_{0}(K_{n},M_{n}(V))$ is
	$\Vert\cdot\Vert_{\infty,n}$-complete. Since $\overline{A_{0}(K_{n}, M_{n}(V))_{sa}}=A_{0}(K_{n})_{sa}$ by Proposition \ref{c1}, we may conclude that 
	$$A_{0}(K_{n}) = \overline{A_{0}(K_{n}, M_{n}(V))} = A_{0}(K_{n}, M_{n}(V))$$ 
	for $A_{0}(K_{n}, M_{n}(V))$ is $\Vert\cdot\Vert_{\infty, n}$-complete.
\end{proof}

\begin{remark}
	Under the assumptions of Theorem \ref{mou}, $L_{n}$ is compact for each $n\in \N.$ To see this, 
	let $\{u_{\alpha}\}$ be a net in $L_{n}.$ Since $L_{n}\subseteq K_{n}$ and $K_{n}$ is compact,
	$u_{\alpha}$ has subnet $\{u_{\beta}\}$ that convergent  $u_{0}\in K_{n}$. Since $e_{n}\in
	A_{0}(K_{n}).$ Therefore $1= e_{n}(u_{\beta}) \to e_{n}(u_{0})$ so that $e_{n}(u_{0})=1.$ Hence
	$u_{0}\in
	L_{n}.$ 
\end{remark}
\begin{proposition}\label{iso}
	$A_{0}(K_{n})$ is order isomorphic to $A(L_{n})$.% via $a\mapsto a|_{L_n}$, $a\in A_0(K_{n}).$ 
\end{proposition}
\begin{proof} 
	It suffices to prove that the map $a\mapsto a|_{L_n}$ from $A_0(K_n)$ into $A(L_n)$ is surjective. Let $a\in A(L_{n}).$ Since $L_{n}$ is convex, there is an affine map $b$ on $K_{n}$ such that $b|_{L_{n}}=a$ and $b(0)=0.$ Now let $u_{\alpha}$ be net in $K_{n}$ such that $u_{\alpha}\to u_{0}$ in $K_{n}$. Since $e_{n}\in A_{0}(K_{n}),$ $e_{n}(u_{\alpha})\to e_{n}(u_{0}).$ By Proposition \ref{c4}, we have $u_{\alpha} = \lambda_{\alpha} \widehat{u_{\alpha}}$ for some $\widehat{u_{\alpha}} \in L_{n}$ and $\lambda_{\alpha} \in [0,1].$ If $u_{0}=0,$ then $\lambda_{\alpha} = \lambda_{\alpha} e_{n}(\widehat{u_{\alpha}}) = e_{n}(u_{\alpha}) \to e(0) = 0.$ Therefore, $b(u_{\alpha}) = \lambda_{\alpha} a(\widehat{u_{\alpha}}) \to 0=b(0).$ Again if $u_{0}\neq 0,$ then by Proposition \ref{c4}, we have $u_{0}=\lambda_{0}\widehat{u_{0}}$ for some $\lambda_{0}\in (0,1]$ and $\widehat{u_{0}} \in L_{n}.$ Then $\lambda_{\alpha} = \lambda_{\alpha} e_{n}(\widehat{u_{\alpha}}) = e_{n}(u_{\alpha}) \to e_{n}(u_{0}) = \lambda_{0}.$ Thus we have $\widehat{u_{\alpha}} \to \widehat{u_{0}}.$ Since $b(u_{\alpha}) = \lambda_{\alpha} a(\widehat{u_{\alpha}}),$ we have $b(u_{\alpha}) \to \lambda_{0} a(u_{0}) = b(u_{0}).$ 
\end{proof}

\begin{remark} In Proposition \ref{iso}, we note that $a \mapsto a|_{L}$ is an isometry from $A_{0}(K_{n})_{sa}$ onto $A(L_{n})$ as well. Hence  $(A_{0}(K_{1}), e)$ is unitally, complete isometrically, completely order isomorphic to $(A(L_{1}), e)$ as matrix order unit spaces.
\end{remark}

%\begin{acknowledgement} 
%	The first author is thankful to the Department of Atomic Energy, Government of India for providing financial support.
%\end{acknowledgement}

\end{document}